\documentclass[9pt,a4paper]{article}
\usepackage[utf8]{inputenc}
\usepackage[T1]{fontenc}
\usepackage{lmodern} 
\usepackage[francais,english]{babel}
\usepackage{amsmath,amsthm}
\usepackage{hyperref}
\hypersetup{
    colorlinks=true,                          
    linkcolor=blue, 
    citecolor=red, 
    urlcolor=blue,  } 
\usepackage{mathrsfs}
\usepackage{bbm} 
\usepackage{amssymb}
\usepackage{amsfonts}
\usepackage{pifont}
\usepackage{stmaryrd}
\usepackage{dsfont}
\usepackage{graphicx,pstricks}
\usepackage{boites,boites_exemples}
\usepackage{enumerate}
\usepackage{empheq}
\usepackage{lscape}
\usepackage[all]{xy}
\input xy
\xyoption{all}
\newcommand{\be}{\begin{enumerate}}
\newcommand{\ee}{\end{enumerate}}
\newcommand{\R}{\mathbb{R}}

\newcommand{\N}{\mathbb{N}}
\newcommand{\C}{\textnormal{C}}
\newcommand{\W}{\textnormal{W}}

\newcommand{\B}{\textnormal{B}}
\renewcommand{\H}{\textnormal{H}}
\newcommand{\Ld}{\textnormal{L}}
\newcommand{\Ll}{\textnormal{L}_{\textnormal{loc}}}

\newcommand{\T}{\mathbb{T}}
\newcommand{\X}{\textnormal{X}}
\newcommand{\V}{\textnormal{V}}
\newcommand{\Y}{\textnormal{Y}}
\newcommand{\Z}{\textnormal{Z}}

\renewcommand{\div}{\textnormal{div}}

\newcommand{\Ldiv}{\Ld_{\textnormal{div}}}
\newcommand{\Hdiv}{\H_{\textnormal{div}}}

\newcommand{\ffi}{\varphi}
\newcommand{\dd}{\mathrm{d}}

\newtheorem{thm}{Theorem}

\newtheorem{coro}{Corollary}
\newtheorem{defi}{Definition}
\newtheorem{propo}{Proposition}
\newtheorem{lem}{Lemma}
\newtheorem{rem}{Remark}

\title{Uniqueness of the solution to the 2D Vlasov-Navier-Stokes system}

\author{Daniel Han-Kwan\footnote{CNRS and CMLS, \'Ecole polytechnique, France}, \'Evelyne Miot\footnote{CNRS and IF, Universit\'e Grenoble-Alpes, France}, Ayman Moussa\footnote{Sorbonne Universités, UPMC Univ. Paris 6 \& CNRS, UMR 7598 LJLL} and Iv\'an Moyano\footnote{DPMMS, University of Cambridge, United Kingdom}}

\begin{document}
\maketitle 

\begin{abstract}
We prove a uniqueness result for weak solutions to the Vlasov-Navier-Stokes system in two dimensions, both in the whole space and in the periodic case, under a mild decay condition on the initial distribution function. The main result is achieved by combining methods from optimal transportation (introduced in this context by G. Loeper) with the use of Hardy's maximal function, in order to obtain some fine Wasserstein-like estimates for the difference of two solutions of the Vlasov equation.
\end{abstract}

\tableofcontents

\section{Introduction}
 Fluid-kinetic systems describe the behavior of a dispersed phase of particles within a fluid. These systems allow a lot of different modelling possibilities for both the fluid and the dispersed phase, typically described by a Vlasov equation, the combination of which leads to challenging mathematical issues. We refer to the introduction of \cite{GHM} for a recent state of the art on the subject. \par 
 
A paradigm of such models is the Vlasov-Navier-Stokes system: among other applications, it has been used to describe the transport of particles in the upperways for medical purposes (see for instance \cite{BGLM}) and at the same time it offers important mathematical issues to deal with: existence of solution \cite{ano-bou,bougramou}, long-time behavior \cite{choikwon,GHM}, asymptotic limit \cite{bendesmou}, controllability \cite{Moy}... \par 

Existence for fluid-kinetic systems is nowadays well-understood, and there is no exception for the Vlasov-Navier-Stokes system: global weak solutions are known to exist for periodic boundary conditions (\cite{bou-des-grand-mou}), bounded domains with various boundary condition (\cite{chengyu,wangyu}) or even moving domains (\cite{bougramou}) and if the initial data is small enough, the system is well-posed in a narrower set of solutions (\cite{choikwon,Moy}). 
In contrast, the uniqueness issue has not been very much studied.
Of course, there is no hope to obtain in 3D a result that might not even hold for the Navier-Stokes system alone, but it is reasonable to hope for uniqueness in 2D, because in that case Leray solutions are indeed unique. The problem is yet challenging because the Vlasov-Navier-Stokes system is strongly coupled: even though both equations taken alone (fluid or kinetic) are well-posed, some effort is needed to recover well-posedness for the whole system. Up to our knowledge, uniqueness of weak solutions has not been proved yet and we intend to fill this gap. 
Our method of proof relies strongly on the regularity properties of the Leray solution in 2D and the optimal transport approach introduced by Loeper for the proof of uniqueness of the Vlasov-Poisson system \cite{Loeper}.
\section{Notations}
The norm of a vector space $X$ will always be denoted $\|\cdot\|_X$, with an exception for the $\textnormal{L}^p$ spaces for which we will often use the notation $\|\cdot\|_p$ if there is no ambiguity. We omit the exponent for the functional spaces constituted of vector fields: we denote for instance  $\textnormal{L}^2(\R^2)$ instead of $\textnormal{L}^2(\R^2)^2$ the set of all vector fields ${\R^2 \rightarrow\R^2}$ whose norm is square-integrable. 

\vspace{2mm}

We denote by $\T^2$ the 2D torus and we will work on $\Omega=\T^2$ or $\Omega=\R^2$. We denote by $\mathscr{D}_{\div}(\Omega)$ the set of smooth $\R^2$ valued vector-fields having $0$ divergence and compact support in $\Omega$ and by $\mathscr{D}(\Omega\times\R^2)$ the set of smooth functions with compact support in $\Omega\times\R^2$ (and similarly for $\mathscr{D}(\R_+\times\Omega\times\R^2)$). The closures of $\mathscr{D}_{\div}(\Omega)$ for the $\Ld^2(\Omega)$ and $\H^1(\Omega)$ are respectively denoted $\Ldiv^2(\Omega)$ and $\Hdiv^1(\Omega)$, and we write $\Hdiv^{-1}(\Omega)$ for the dual of the latter. For a function $h:\R_+ \times \Omega\times\R^2\rightarrow \R_+$ we will use the following notation for $k\in\N$ 
\begin{align*}
m_k h(t,x) &:= \int_{\R^2} h(t,x,v) |v|^k \,\dd v,\\
M_k h(t) &:= \int_{\Omega} m_k h(t,x)\,\dd x.
\end{align*}

\vspace{2mm} 

We will often use the notation $A\lesssim B$ to say $A\leq C B$ where $C$ is some constant which does not depend on any variable defining either $A$ or $B$.

\section{The Vlasov-Navier-Stokes system and main result}
\indent The purpose of this paper is to study uniqueness of  weak solutions of the Vlasov-Navier-Stokes, i.e., for pairs $(u,f)$, where $u=u(t,x)$ is a velocity field and $f=f(t,x,v)$ is a distribution function, satisfying the system
\begin{align*}
\partial_t u + (u\cdot \nabla_x)u -\Delta_x u + \nabla_x p &= j_f-\rho_f u, \\
\div_x u & = 0, \\
\partial_t f + v \cdot \nabla_x f + \div_v [f(u-v)] &=0,\\
u(0,\cdot) &= u_0(\cdot),\\
f(0,\cdot,\cdot) &= f_0(\cdot,\cdot),
\end{align*}
studied on $\R_+\times\Omega$ and $\R_+\times \Omega\times\R^2$, 
for $\Omega= \T^2$ or $\R^2$,
where 
\begin{align*}
\rho_f(t,x) &:= \int_{\R^2} f(t,x,v)\,\dd v, \\
j_f(t,x) &:= \int_{\R^2} vf(t,x,v)\,\dd v.
\end{align*}

\subsection{Comments on the existence of solutions}
In the case $\Omega=\T^2$, we infer directly from  \cite{bou-des-grand-mou}\footnote{Strictly speaking the analysis of \cite{bou-des-grand-mou} considers the 3D case $\T^3$, which extends trivially to $\T^2$.} the existence of global weak solutions for the former system, when the initial data $(u_0,f_0)$ satisfies : 
\begin{align}
\label{eq:initu}u_0 &\in\Ldiv^2(\Omega), \\
\label{eq:initf}0\leq f_0 &\in\Ld^1 \cap\Ld^\infty(\Omega\times\R^2),\\
\label{eq:initfbis}(x,v)\mapsto f_0(x,v)|v|^2 &\in\Ld^1(\Omega\times\R^2).
\end{align}
The case $\Omega=\R^2$ has not been written explicitely in the literature, but one can easily adapt the proof of the periodic framework, and recover global existence of weak solution for initial data verifying \eqref{eq:initu} -- \eqref{eq:initfbis}. For the sake of completeness, we give in Appendix A a brief justification of this statement. The typical solutions that one can build by this procedure have the following behavior : the velocity field $u=u(t,x)$ satisfies the Navier-Stokes equation in the weak sense (Leray), whereas the distribution function $f=f(t,x,v)$ is a weak (renormalized) solution of the corresponding transport equation. More precisely, we shall work within the following framework.

\begin{defi}\label{def:sol}
Consider $(u_0,f_0)$ satisfying \eqref{eq:initu} -- \eqref{eq:initfbis}. A solution of the Vlasov-Navier-Stokes system with initial condition $(u_0,f_0)$ is a pair $(u,f)$ such that
\begin{align*}
u &\in \Ll^\infty(\R_+;\Ld^2(\Omega))\cap\Ll^2(\R_+;\H^1_\div(\Omega))\cap \W^{1,2}_{\textnormal{loc}}(\R_+;\H^{-1}_\div(\Omega)),\\
f&\in  \Ll^\infty(\R_+;\Ld^1\cap\Ld^\infty(\Omega\times\R^2)),\\
j_f &\in\Ll^1(\R_+\times\Omega),\\
\rho_f u &\in \Ll^1(\R_+\times\Omega),
\end{align*}
and such that for any test functions $(\psi,\phi)\in\mathscr{C}^\infty(\R_+;\mathscr{D}_\div(\Omega))\times\mathscr{D}(\R_+ \times\Omega\times\R^2)$,  the following identities hold for all $t\geq 0$:
\begin{align*}
 \int_{0}^t \int_{\Omega} u \cdot[\partial_t \psi + (u\cdot \nabla) \psi + \Delta \psi] +\int_0^t \int_{\Omega}(j_f-\rho_f u)\cdot \psi &= \langle  u(t), \psi(t)\rangle - \int_{\Omega} u_0\cdot  \psi(0)  \\
  \int_{\R_+} \int_{\Omega}\int_{\R^2} f[\partial_t \phi+ v\cdot\nabla_x \phi + (u-v)\cdot \nabla_v \phi] &=  -\int_{\Omega}\int_{\R^2} f_0 \phi(0),
\end{align*}
where the duality bracket in the r.h.s. of the first equation makes sense because one has in particular that $u\in\mathscr{C}^0(\R_+;\Hdiv^{-1}(\Omega))$.
\end{defi}

\subsection{Main result: uniqueness of weak solutions}

The uniqueness theorem that we are about to state and prove can be applied to the weak solutions described in the previous section, provided that the initial distribution function $f_0$ satisfies, in addition to \eqref{eq:initu} -- \eqref{eq:initfbis}, a certain decay estimate. More precisely, we shall impose the following.
\begin{defi}\label{def:adm}
The pair  $(u_0,f_0)$ is an \emph{admissible} initial datum if  \eqref{eq:initu} -- \eqref{eq:initfbis} hold, along with the following condition :
\begin{equation}
\label{assu:lp}
\forall R>0,\qquad\Gamma_R(f_0):(t,v) \mapsto (1+|v|^2)\sup_{\C^R_{t,v}} f_0 \in \Ll^\infty(\R_+;\Ld^1(\R^2)),
\end{equation}
where $\C^R_{t,v}:=\Omega \times\B(e^t v,R)$.
\end{defi}
\begin{rem}
Condition \eqref{assu:lp} is reminiscent of the one introduced in \cite{lionsperth} for the study of the Vlasov-Poisson system. Observe that the exponential factor appearing here comes from the dissipation term $-v$ in the drag force (see (\ref{eq:char1}) and (\ref{eq:char2})).
\end{rem}

\begin{rem}
A sufficient condition for Assumption \eqref{assu:lp} to hold is the following pointwise decay condition: 
\begin{equation*}
\exists q>4 \quad \textrm{ such that }  \quad   f_0(x,v)  \lesssim \frac{1}{1+|v|^q},\quad \forall (x,v)\in \Omega\times \R^2.
\end{equation*}
Indeed, in that case
\begin{align*}
\sup_{\C^R_{t,v}} f_0 \lesssim \frac{1}{1+|e^t |v|-R|^q},
\end{align*}
and one checks directly that 
\begin{align*}
\int_{\R^2} \frac{1+|v|^2}{1+|e^t |v|-R|^q} \dd v < + \infty.
\end{align*} Since the previous integral continuously depends on $t\in[0,T]$, condition (\ref{assu:lp}) follows.
\end{rem}

\medskip

We are now in position to state our main result:

\begin{thm}\label{thm:uni}
  Consider an admissible initial datum $(u_0,f_0)$ in the sense of Definition \ref{def:adm}. Then the corresponding Vlasov-Navier-Stokes system has a unique weak solution in the sense of Definition \ref{def:sol}.
\end{thm}

Since the cases $\Omega= \T^2$ or $\R^2$ 
are almost identical, we shall focus on the case $\Omega=\R^2$. The only required modification whenever $\Omega = \T^2$ is explained in Appendix B.

\subsection{Strategy and outline}

\subsubsection{Strategy of the proof}
In order to prove Theorem \ref{thm:uni}, we shall consider two weak solutions of the Vlasov-Navier-Stokes system, which must eventually coincide. After exhibiting several \emph{a priori} estimates, we follow a strategy based on four points.  

\begin{enumerate}

\item In the first step, we exploit the usual energy estimate for the Navier-Stokes system, which allows to identify some key terms to be controlled otherwise. 

\item In the second step, we introduce a functional to compute, in a weak sense, the distance between the two solutions of the transport equation. 

\item In third step we make a key use of the Hardy's maximal function to show that the problematic terms can be handled  with the functional introduced in the previous step. 

\item Finally, in the fourth step, a Gronwall-like argument allows to conclude that the two solutions must coincide, as the distance between the solutions of the transport equation vanishes.

\end{enumerate}

We give a more detailed description of this road map in Section \ref{sec: strategydetailed}, once all the ingredients and tools have been introduced.

\subsubsection{Outline of the paper}

Section \ref{sec: thmuni} is devoted to the proof of our main result. To do this, we introduce in Section \ref{sec: maintools} the main tools and ideas needed in the proof. In Section \ref{sec: aprioriestimates} we derive some energy estimates for the weak solutions of the Vlasov-Navier-Stokes system, especially concerning the propagation of kinetic moments and a useful energy-dissipation inequality for the fluid. This allows to develop in Section \ref{sec: strategydetailed} the main steps of the proof of Theorem \ref{thm:uni}. 
In Section \ref{sec: commentsperspectives}, we gather some comments and perspectives. 
In Appendix \ref{sec: appendixWeak} we give a sketch of the proof of existence of weak solutions in the whole space $\R^2$. 
Finally, in Appendix \ref{sec: AppendixTorus} we make some comments on the adaptation of our main result to the 2D torus.

\section{Proof of Theorem \ref{thm:uni}}
\label{sec: thmuni}
In order to simplify the presentation, we shall assume from now on that 
\begin{align}\label{ass:f0norm}
\int_{\R^2}\int_{\R^2} f_0(x,v)\,\dd x\,\dd v = 1.
\end{align}
This normalization does not really affect the structure of the proof but lighten several computations. We will however point out the only two spots in which one should pay attention when this normalization does not hold (see Remarks \ref{rem:normalized} and \ref{rem:normalized2} for details).

\subsection{Some of the main tools}
\label{sec: maintools}

The goal of this section is to highlight three specific tools needed in our proof. 

\subsubsection{Log-Lipschitz regularity for the Navier-Stokes system}

The first one concerns the log-Lipschitz regularity of solutions to the 2D Navier-Stokes equation and its very useful consequences on the associated flow. 
\begin{thm}[Chemin-Lerner]
\label{thm:flow}
Consider $u_0\in\Ld^2(\R^2)$ such that $\div_x u_0=0$ and $F\in\Ll^2(\R_+;\Ld^2(\R^2))$. Then the unique Leray solution $u$ of the Navier-Stokes equation initialized with $u_0$ and with source term $F$ belongs to $\Ll^1(\R_+;\mathscr{C}^0(\R^2))$ and satisfies that for any $\eta \in (1/2,1]$ and for any $x,y \in \R^2$ with $|x-y|\leq 1/e$, 
\begin{align*}
|u(t,x)-u(t,y)| \leq \gamma_\eta(t) |x-y| (\log|x-y|^{-1})^\eta,
\end{align*}
for some function $\gamma_\eta\in\Ll^1(\R_+)$.
\end{thm}

\begin{rem}
This is a slight generalization of a result of Chemin and Lerner established in \cite{chemlen}. In their article, there is no source term in the Navier-Stokes equation, but the analysis can be straightforwardly adapted to add a source with $\Ld^2$  regularity.

\end{rem}

\subsubsection{The Loeper's functional}

The second tool is not a turnkey result but rather a method. In his proof of uniqueness for the Vlasov-Poisson system, Loeper introduced in \cite{Loeper}  a functional reminiscent of optimal transportation theory to estimate the distance between two hypothetical measure-valued solutions to the Vlasov equation. We use the very same functional in our proof (see (\ref{eq:Loeper}) for details).

\subsubsection{The Hardy's maximal function}

 The third and last tool is the Hardy's maximal function. Among other numerous applications, the maximal function has been used by Crippa and De Lellis \cite{Crippa-Delellis} to prove uniqueness of the flow associated to Sobolev vector 
 fields. \par 

In this work, we use the following definition\footnote{
For the torus case, we refer to Appendix B.}: for every $g\in \Ld^1_\text{loc}(\R^2)$ the associated maximal function, denoted $Mg$, is defined by 
\begin{align*}
M g(x):=\sup_{r>0}\frac{1}{|\B_r(x)|}\int_{\B_r(x)}|g|(y)\,\dd y, \quad \quad \textrm{for a.e. } x \in \R^2.
\end{align*}
We will rely on the following result, which is a (very) particular case of some of the remarkable properties of the maximal function (see e.g. \cite{Stein} or \cite[Lemma 3]{AcerbiFusco}). 
\begin{propo}\label{propo:max}
 If $g\in\Ld^2(\R^2)$ then so is $M g$. Furthermore,
\begin{align}\label{ineq:max}
\forall g \in \Ld^2(\R^2),\qquad \|M g\|_2 \lesssim \|g\|_2.
\end{align}
Moreover, if $g\in\H^1(\R^2)$ then, for a.e. $x,y\in\R^2$ one has (with a constant independent of $g$)
\begin{align}\label{ineq:max:2}
|g(x)-g(y)|\lesssim |x-y|(M |\nabla g|(x)+M |\nabla g|(y)).
\end{align}
\end{propo}

\subsection{Some \emph{a priori} estimates}
\label{sec: aprioriestimates}

First of all we use the DiPerna-Lions theory (see \cite{DPL}) for linear transport equations and some properties of the Navier-Stokes equation in 2D to prove that solutions to the Vlasov-Navier-Stokes system in the sense of Definition \ref{def:sol} are in fact more workable than they appear. Consider thus such a solution $(u,f)$ in the sense of Definition \ref{def:sol}.

\subsubsection{Propagation of moments}\label{subsubsec:moments}
The Sobolev regularity of $u$ and the integrability of $f$ impose that the latter is actually the renormalized solution of 
\begin{align}
\partial_t f + v \cdot \nabla_x f  + (u-v)\cdot \nabla_v f - 2 f &= 0,  \label{eq:renormalizedVlasov} \\
f(0,\cdot,\cdot) &= f_0. \nonumber
\end{align}
 In particular, $f\in\mathscr{C}^0([0,T];\Ld^p(\R^2))$ for all $p<\infty$ and the weak formulation of the Vlasov equation can thus be extended into the following: for any $\phi\in\mathscr{C}^\infty(\R_+;\mathscr{D}(\R^2\times\R^2))$, for any $t\in\R_+$,
\begin{align*}
  \int_{0}^t \iint_{\R^2 \times \R^2} f[\partial_t \phi+ v\cdot\nabla_x \phi + (u-v)\cdot \nabla_v \phi] &= \iint_{\R^2 \times \R^2} f(t) \phi(t) - \iint_{\R^2 \times \R^2} f_0 \phi(0).
  \end{align*} Observe that, compared to the weak formulation given in Definition \ref{def:sol}, the equality above allows to take into account the initial distribution $f_0$. \par

\begin{lem}\label{lem:interp}
Let $(u,f)$ be a weak solution of the Vlasov-Navier-Stokes system according to Definition \ref{def:sol}. Then the identity
\begin{equation}
M_2 f(t)  + 2 \int_0^t M_2 f(s) \dd s = M_2 f_0 + 2\int_0^t \int_{\R^2} u \cdot j_f(s,x) \,\dd x\,\dd s \label{eq:M2}
\end{equation} holds for any $t\geq 0$.
\end{lem}

\begin{proof}
Since $f$ is the renormalized solution of (\ref{eq:renormalizedVlasov}), $f$ is also the strong limit (in the sense of $\mathscr{C}^0([0,T];\Ld^p(\R^2))$ for $p<\infty$) of any sequence $(f_n)_n$ solving the Vlasov equation with a regularized version $(u_n)_n$ of the vector field $u$, that approaches it in $\Ll^2(\R_+;\H^1(\R^2))$, and compactly supported initial datum $(f_{0,n})_n$ (approaching $f_0$ adequately). For all $n$, $f_n$ is known explicitely (\emph{via} characteristics curves), for all $t$, $f_n(t)$ is compactly supported and $(f_n)_n$ is bounded in $\Ll^\infty(\R_+;\Ld^\infty(\R^2\times\R^2))$. Then, after multiplication of the Vlasov equation by $|v|^2$ and suitable integration by parts, one gets 
\begin{equation}
M_2 f_n(t) + 2 \int_0^t M_2 f_n(s) \dd s \leq M_2 f_0 +  2\int_0^t \int_{\R^2} |u_n| m_1 f_n(s,x)\,\dd x \,\dd s. \label{ineq:M2fn} 
\end{equation} Let us recall the following standard interpolation estimate (see e.g. \cite{bou-des-grand-mou} for a proof):
\begin{equation}
\|m_\ell h \|_{\frac{k+2}{\ell+2}} \lesssim \|h\|_\infty^{\frac{k-\ell}{k+2}} M_k h(t)^{\frac{\ell+2}{k+2}}, \quad \forall \ell, k \in \N \textrm{ with } \ell  \leq k,
\label{ineq:interp}
\end{equation} for any $h:\R_+ \times \R^2 \times \R^2 \rightarrow \R_+$. In particular, taking $\ell=1$ and $k=2$ in (\ref{ineq:interp}) with the choice $h=f_n$, we infer the following inequality on any fixed interval $[0,T]$ 
\begin{align*}
\|m_1 f_n(t)\|_{4/3} \lesssim M_2 f_n(t)^{3/4},
\end{align*}
as $(f_n)_n$ is bounded in $\Ll^\infty(\R_+;\Ld^\infty(\R^2\times\R^2))$. Since $(u_n)_n$ is bounded in $\Ll^2(\R_+;\H^1(\R^2))$, it is bounded in $\Ll^2(\R_+;\Ld^p(\R^2))$ for all finite values of $p$ by Sobolev embedding: the last estimate can be employed in \eqref{ineq:M2fn} together with the Gronwall lemma to prove that $(M_2 f_n)_n$ is bounded in $\Ll^\infty(\R_+)$. Thus, the strong convergence of $(f_n)_n$ toward $f$ and Fatou's lemma allow to prove that $(t,x,v)\mapsto |v|^2 f \in \Ll^\infty(\R_+;\Ld^1(\R^2\times\R^2))$. This stronger integrability allows to extend further the weak formulation. In particular, using $|v|^2$ as a test function by an approximation argument, we deduce (\ref{eq:M2}).

\end{proof}

We prove next that $f$ satisfies even stronger moment estimates, thanks to the extra assumption \eqref{assu:lp}.

\begin{propo}\label{prop:lp}
Consider $(u_0,f_0)$ an admissible initial datum in the sense of Definition \ref{def:adm} and let $(u,f)$ be a solution of the corresponding Vlasov-Navier-Stokes system in the sense of Definition \ref{def:sol}. Assume furthermore  that $u$ satisfies 
\begin{equation}
u\in\Ll^1(\R_+;\Ld^\infty(\R^2)).
\label{eq:champdevitessesLINFINIenEspace}
\end{equation} Then the following estimate holds for $t\in [0,T]$ 
\begin{align}
\label{ineq:momlp}\|m_0 f(t)\|_\infty + \|m_1 f(t)\|_\infty + \|m_2 f(t)\|_\infty \leq e^{2T}\sup_{[0,T]} \|\Gamma_R(f_0)(t,\cdot)\|_{\Ld^1(\R^2)},
\end{align}
where $\Gamma_R(f_0)$ is the function defined in \eqref{assu:lp} and $R=  e^T\|u\|_{\Ld^1(0,T;\Ld^\infty(\R^2))}$.
\end{propo}

\begin{rem}
The additional control on the moments of $f$ given by (\ref{ineq:momlp}) is crucial in our proof of uniqueness. On the other hand, it relies on the \emph{a priori} estimate (\ref{eq:champdevitessesLINFINIenEspace}), which seems to require extra regularity on velocity field $u$. However, this does not represent any limitation, as any weak solution of the Navier-Stokes system under a sufficiently regular force satisfies (\ref{eq:champdevitessesLINFINIenEspace}). We prove this in detail in Lemma \ref{raduc}.
\end{rem}

 \begin{proof}
Owing to the (strong) stability of renormalized solutions, we can without loss of generality assume that $u$ is smooth: it suffices to approach it by an approximating sequence $(u_n)_n$ such that $\|u_n\|_{\Ld^1(0,T;\Ld^\infty(\R^2))} \leq \|u\|_{\Ld^1(0,T;\Ld^\infty(\R^2))}$. In particular,  $f$ is known explictely through the characteristics curves. More precisely, we have the formula:
\begin{equation}
 \label{eq:char}
f(t,x,v)=e^{2t}f_0(\X(0;t,x,v), \V(0;t,x,v)),\end{equation}
where the characteristics $(\X, \V)$ solve (here the dot means derivative along the first variable)
\begin{empheq}[left =\empheqlbrace]{align}
\label{eq:char1}&\dot{\X}(s;t,x,v) = \V(s;t,x,v),\\
& \dot{\V}(s;t,x,v) = u(s,\X(s;t,x,v))-\V(s;t,x,v),  \label{eq:char2}\\
&\X(t;t,x,v)=x,\\
\label{eq:char4}&\V(t;t,x,v)=v.
\end{empheq}
We have in particular for $0\leq s\leq t\leq T$, 
\begin{align*}
|e^t v -e^s \V(s;t,x,v)|\leq \int_s^t e^{\tau} |u(\tau,\X(\tau;t,x,v)| \dd \tau \leq e^T  \|u\|_{\Ld^1(0,T;\Ld^\infty(\R^2))},
\end{align*}
so that $(X(0;t,x,v),\V(0;t,x,v)) \in \C^R_{t,v}:=\R^2\times \B(e^t v,R)$, where we have set $R:=e^T \|u\|_{\Ld^1(0,T;\Ld^\infty(\R^2))}$. 
Using \eqref{eq:char}, the conclusion follows immediately.
\end{proof}

\subsubsection{Energy equality}
In 2D, the incompressible Navier-Stokes system is well-posed if the source term is in $\Ll^2(\R_+;\H^{-1}(\R^2))$ (see for instance Theorem V.1.4 of \cite{boyer-fabrie}) and the corresponding Leray solution is known to satisfy the energy equality. In our coupling the source term of the fluid equation is $j_f - \rho_f u$. By duality and Sobolev embedding, $\H^{-1}(\R^2)$ contains all $\Ld^p(\R^2)$ (for finite values of $p$) so that using Lemma \ref{lem:interp} and (\ref{assu:lp}), one recovers that $j_f-\rho_f u$ belongs to $\Ll^2(\R_+;\H^{-1}(\R^2))$. Thus, our solution $u$ is in fact the Leray solution, which therefore belongs to $\mathscr{C}^0(\R_+;\Ld^2(\R^2))$ and satisfies the following identity 
\begin{align*}
\frac{1}{2}\|u(t)\|_2^2 + \int_0^t \|\nabla u(s)\|_2^2 \,\dd s = \frac{1}{2}\|u_0\|_2^2 + \int_0^t \int_{\R^2} u\cdot (j_f- \rho_f u)(s,x)\,\dd x\,\dd s.
\end{align*}
Adding the previous equality to $\frac{1}{2}\times$\eqref{eq:M2} leads to the following important energy-dissipation identity 
\begin{lem}
\label{energy}
Any weak solution $(u,f)$ to the 2D Vlasov-Navier-Stokes system satisfies that, for all $t\in\R_+$, one has 
\begin{equation*}
 \frac{1}{2}\| u(t)\|_2^2 + \frac{1}{2}M_2 f(t) + \int_0^t  \| \nabla u(s) \|_2^2 \,\dd s +   \int_0^t \int_{\R^2}\int_{\R^2}  f |v-u|^2 \,\dd v\,\dd x \, \dd s  =  \frac{1}{2}\|u_0\|_{2}^2  + \frac{1}{2} M_2 f_0.
\end{equation*}
\end{lem}

\subsubsection{An estimate \emph{à la} Chemin-Gallagher}
In order to invoke Proposition \ref{prop:lp}, we need to prove (\ref{eq:champdevitessesLINFINIenEspace}). This is the purpose of the next lemma, inspired in the result ~\cite[Theorem 3]{ChemGal} due to Chemin and Gallagher.
\begin{lem}
\label{raduc}
Any global weak solution $(u,f)$ to the 2D Vlasov-Navier-Stokes system satisfies
\begin{equation}
u \in\Ll^{2}(\R_+;\Ld^\infty(\R^2)).
\label{eq:champdevitessesLINFINIplusL2enTEMPS}
\end{equation}
\end{lem}

\begin{rem}
Observe that we prove a stronger condition that (\ref{eq:champdevitessesLINFINIenEspace}). Indeed, (\ref{eq:champdevitessesLINFINIplusL2enTEMPS}) gives better integrability in time, which will be needed afterwards in the proof of the uniqueness result. 
\end{rem}

\begin{proof} Let $T>0$. Using (\ref{ineq:interp}) with $k=2$ and $\ell=0,1$ we infer for a.e. $t \geq 0$ and $x \in \R^2$,
\begin{align*}
|\rho_f(t,x) |  &\leq \| f \|_{\Ld^\infty(0,T; \Ld^\infty(\R^2 \times \R^2))}^{1/2} (M_2 f)^{1/2},\\
|j_f(t,x) |  &\leq \| f \|_{\Ld^\infty(0,T; \Ld^\infty(\R^2 \times \R^2))}^{1/4} (M_2 f)^{3/4},
\end{align*}
from which, combined with Lemma~\ref{energy}, we deduce 
\begin{align}
\label{rho2}
\rho_f &\in \Ld^\infty(0,T; \Ld^2(\R^2)),\\
\label{j43}
j_f &\in \Ld^\infty(0,T; \Ld^{4/3}(\R^2)).
\end{align}
We have also by Lemma~\ref{energy} that
$$
\nabla u \in  \Ld^2(0,T; \Ld^2(\R^2)),
$$
and thus by Sobolev embedding, we obtain that for all $p<+\infty$,
$$
u \in  \Ld^2(0,T; \Ld^p(\R^2)).
$$
Thanks to~\eqref{rho2}, we have that  for all $q<2$,
$$
\rho_f u \in  \Ld^2(0,T; \Ld^q(\R^2)).
$$
Furthermore, combining with the information
$$
u \in  \Ld^\infty(0,T; \Ld^2(\R^2)).
$$
obtained from Lemma~\ref{energy}, we can also get by interpolation that  for all $p,r<+\infty$
$$
u \in  \Ld^r(0,T; \Ld^p(\R^2)).
$$
We deduce that for all $p,r < 2$, 
$$
u\cdot \nabla u \in  \Ld^p(0,T; \Ld^r(\R^2)).
$$
Therefore, we can see the Navier-Stokes equation as a Stokes equation with a source $F$ given by
$$
F := j_f - \rho_f u  - u \cdot \nabla u 
$$
belonging to $ \Ld^{4/3}(0,T; \Ld^{4/3}(\R^2))$. We next write $u=S(t)u_0+w$, where $S(t)$ is the semi-group associated to the heat equation 
and $w\in \Ld^\infty(0,T;\Ld^2(\R^2))$ solves
$$
\partial_t w -\Delta w +\nabla p=F, \quad \div w = 0, \quad w(0)=0.
$$ 
We apply maximal estimates for the Stokes problem on $\R^2$ with zero initial value, which yield that 
$$
D^2 w \in  \Ld^{4/3}(0,T; \Ld^{4/3}(\R^2)).
$$
We use the Gagliardo-Nirenberg inequality stating that for any $\ffi\in\mathscr{D}(\R^2)$
\begin{align*}
\|\ffi\|_\infty \lesssim \|\ffi\|_2^{1/3} \| D^2\ffi\|_{4/3}^{2/3}, 
\end{align*}
which gives 
$$
w \in {\Ld^2(0,T; \Ld^{\infty}(\R^2))}.
$$
On the other hand, for what concerns the contribution of the initial data, we argue as in~\cite[Theorem 3]{ChemGal} and obtain
$$\|S(t) u_0\|_{\Ld^2(0,T;\Ld^\infty(\R^2))}\lesssim  {\|u_0\|_{L^2}}.$$
(We have used the heat semi-flow characterization of the Besov space $\B^{-1}_{\infty, 2}(\R^2)$, see e.g. \cite[Theorem 5.3]{Lemarie} and the continuous 
embedding of $\Ld^2(\R^2)$ into $\B^{-1}_{\infty, 2}(\R^2)$, see e.g. \cite[Proposition 2.71]{BCD}.)
This implies the  claimed result.
\end{proof}
Gathering Lemma \ref{raduc} and Proposition \ref{prop:lp}, we infer:
\begin{coro}
\label{coro:lpeve2}
The conclusions of Proposition \ref{prop:lp}  hold
for any weak solution $(f,u)$ to the 2D Vlasov-Navier-Stokes system with admissible initial datum.
\end{coro}
\subsubsection{Local log-lipschitz regularity for $u$} 
As already explained, we will need for $u$ the regularity stated in Theorem~\ref{thm:flow}. We introduce the following (non-decreasing and concave) modulus of continuity  
\begin{equation}
\label{Psi}
\Psi:\tau\mapsto \tau |\log \tau|\quad \text{if }\tau\in [0,e^{-1}],\quad \Psi(\tau)=e^{-1}\quad \text{if }\tau\in [e^{-1},+\infty).
\end{equation}
Relying on Lemma \ref{raduc} and on Corollary \ref{coro:lpeve2}, we can prove the following lemma.
\begin{lem}
\label{raducbis}
Any global weak solution $(f,u)$ to the 2D Vlasov-Navier-Stokes system satisfies, for $|x-y|\leq 1/e$ 
\begin{align*}
|u(t,x)-u(t,y)| \leq \gamma(t) \Psi(|x-y|)
\end{align*}
for some function $\gamma\in\Ll^1(\R_+)$.
\end{lem}
\begin{proof}
In order to apply Theorem~\ref{thm:flow}, it suffices to establish that the source term
$j_f-\rho_f u$ belongs to $\Ld^2(0,T; \Ld^{2}(\R^2))$. 

By Corollary \ref{coro:lpeve2}, we deduce that $j_f$ and $\rho_f$ belong to $\Ld^\infty(0,T; \Ld^{\infty}(\R^2))$. Since $u$ belongs to  $\Ld^\infty(0,T; \Ld^{2}(\R^2))$, the quantity $\rho_f u$ belongs to  $\Ld^\infty(0,T; \Ld^{2}(\R^2))$. Furthermore, by~\eqref{j43}, $j_f$ belongs to $\Ld^\infty(0,T; \Ld^{4/3}(\R^2))$ so that  by interpolation, it belongs to  $\Ld^\infty(0,T; \Ld^{2}(\R^2))$. Finally, $j_f-\rho_f u$ belongs to $\Ld^\infty(0,T; \Ld^{2}(\R^2))$, and the conclusion follows.
\end{proof}

\subsection{The proof in four steps}
\label{sec: strategydetailed}

We can now proceed with the actual proof of Theorem \ref{thm:uni}. We do so by proving uniqueness on any interval $[0,T]$, in four steps. Step 1 identifies the key  term which appears when one simply tries to follow the usual energy estimate for the Navier-Stokes system (that is, the computation that leads to uniqueness of Leray solutions in 2D). Since this term cannot simply be controlled by the fluid energy or it dissipation, another functional has to be introduced to measure the distance between the two kinetic parts of the solutions. This functional $Q$ is introduced in Step 2 (it is the one used by Loeper in \cite{Loeper}). To avoid the ``rob Peter to pay Paul'' scenario, we check also in Step 2 that the evolution of $Q$ does not produce any extra terms that we cannot handle up to this point. This behavior is ensured up to a positive time $T_0$ which depends only on the $\Ld^2(0,T;\Ld^\infty(\R^2))$ norms of the two fluid solutions. In Step 3 we use the Hardy maximal function to justify the utility of $Q$: the key term identified in Step 1 can be controlled in terms of $Q$ and the fluid's energy and dissipation, on the whole interval $[0,T_0]$. Eventually in Step 4 we use a Gronwall-like argument to prove that uniqueness holds on $[0,T_0]$ and that it is propagated on the whole $[0,T]$.

\vspace{2mm} 

We consider thus $T>0$ and two solutions $(u_1,f_1)$ and $(u_2,f_2)$ of the Vlasov-Navier-Stokes system, in the sense of Definition \ref{def:sol} with an admissible initial datum $(u_0,f_0)$ in the sense of Definition \ref{def:adm}. In the proof, the notation $C$ will refer to a universal constant and the notation $K$ to a constant depending only on $T$ and on the norms $\|u_j\|_{\Ld^2(0,T;\H^1(\R^2))}$ and $\|u_j\|_{\Ld^2(0,T;\Ld^\infty(\R^2))}$ for $j=1,2$ (note that these norms are finite thanks to Lemma~\ref{energy} and Lemma~\ref{raduc}). The constants $C$ and $K$ may  possibly change value from one line to another. Similarly, we will denote by $\gamma$ an nonnegative element of $\Ld^1(\R_+)$ which satisfies a.e. 
\begin{align*}
\max_{k=1,2} \Big(\|u_k(t)\|_\infty + \|\nabla u_k(t)\|_2 + \gamma_k(t)\Big) \leq \gamma(t),
\end{align*}   
where $\gamma_k$ is the $\Ll^1(\R_+)$ element given by Lemma \ref{raducbis} for $u_k$. Again, $\gamma$ may change of value from line to line.
\vspace{2mm}

Before going on with the proof we notice that thanks to Corollary \ref{coro:lpeve2}, using 
$$R=  e^T(\|u_1\|_{\Ld^1(0,T;\Ld^\infty(\R^2))}+\|u_2\|_{\Ld^1(0,T;\Ld^\infty(\R^2))}),$$
we have the following crucial estimate for $k=1,2$ for $t\in[0,T]$:
\begin{align}\label{ineq:crucial}
\|m_0 f_k(t)\|_\infty + \|m_1 f_k(t)\|_\infty + \|m_2 f_k(t)\|_\infty \leq K.
\end{align}

\subsubsection{Step 1: $\Ld^2$ energy estimate for the Navier-Stokes equation}


Since $(u_1,f_1)$ and $(u_2,f_2)$ are two solutions, 
$w:=u_1-u_2$ solves
\begin{align}
\label{eq:w}  \partial_t w + (u_1 \cdot \nabla) w - \Delta w + \nabla p = - (w \cdot \nabla) u_2 - \rho_1 w + (\rho_2-\rho_1) u_2 + j_1 - j_2 ,
\end{align}
where we denote for $k=1,2$, $(\rho_k,j_k) := (\rho_{f_k}, j_{f_k})$.

\medskip

Owing to the regularity of the Leray solutions, we can take $w$ as test function in \eqref{eq:w} and recover, after the usual integration by parts for the Navier-Stokes system (note that $w(0)=0$) for all $t\in[0,T]$
\begin{multline}
\label{eq:energy} \frac{1}{2}\|w(t)\|_2^2+\int_0^t \|\nabla w(s)\|_2^2\,\dd s\\
= -\int_0^t \int_{\R^2}w \cdot (w\cdot \nabla) u_2 \,\dd x\, \dd s -\int_0^t \int_{\R^2}\rho_1 |w|^2 \,\dd x\dd s+ \int_0^t A(s)\,\dd s,
\end{multline}
where we have set
\begin{align}
\label{def-A}
A(s):=\int_{\R^2} (\rho_2-\rho_1) u_2\cdot w \,dx + \int_{\R^2} (j_1-j_2)\cdot w \,\dd x
\end{align}

We can easily handle the first two terms of the r.h.s. in \eqref{eq:energy}. Indeed the first one is in fact  present even without the coupling with the kinetic part of the system, and is handled in a classical way thanks to the Gagliardo-Nirenberg inequality stating that for any $\ffi\in\mathscr{D}(\R^2)$
\begin{align}
\label{ineq:gngen} \|\ffi\|_q \lesssim \|\ffi\|_2^{2/q} \|\ffi\|_{\H^1(\R^2)}^{1-2/q},
\end{align}
which gives for $q=4$
\begin{align}
\label{ineq:gn} \|\ffi\|_4 \lesssim \|\ffi\|_2^{1/2} \| \ffi\|_{\H^1(\R^2)}^{1/2},
\end{align}
so that
\begin{align*}
\| w \cdot (w\cdot \nabla) u_2 \|_1 &\leq \|w\|_4^2 \|\nabla u_2\|_2 \\
&\lesssim  \|w\|_2 \|w\|_{\H^1(\R^2)} \|\nabla u_2\|_2 \\
&\leq C\|w\|_2^2 (\|\nabla u_2\|_2^2+1) + \frac{1}{4} \|w\|_{\H^1(\R^2)}^2.
\end{align*}
We observe here that the worst term $\|\nabla w\|_2^2$ is absorbed by the l.h.s. in \eqref{eq:energy}, while the other opens the way for an eventual use of a Gronwall type estimate because both  $u_k$ belong to $\Ll^2(\R_+;\H^1(\R^2))$. For what concerns the second term in the r.h.s. of \eqref{eq:energy}, it is in fact a gain in our computation because $\rho_1\geq 0$, so that the core of the proof is devoted to the understanding of $A(s)$, since we actually obtained 
\begin{align}\label{ineq:w:gron}
\frac{1}{2}\|w(t)\|_2^2 + \frac{3}{4}\int_0^t \|\nabla w(s)\|_2^2\,\dd s \leq C\int_0^t \gamma(s)\|w(s)\|_2^2 \dd s + \int_0^t A(s) \,\dd s,
\end{align}
for some $\gamma\in\Ll^1(\R_+)$.

\subsubsection{Step 2: Wasserstein-like estimate for the Vlasov equation}

To handle the term $A(t)$, we follow the approach of Loeper \cite{Loeper} and introduce a functional $Q$ inspired from optimal transport to measure the distance between $f_1$ and $f_2$. 
By Lemma~\ref{raducbis}, we observe that both $u_k$ have a sufficient Osgood regularity so as to define the characteristics curves $\Z_k^t(x,v):=(\X_k^t(x,v),\V_k^t(x,v))$ associated to the vector field $(t,x,v) \mapsto (v,u_k(t,x)-v)$ (as defined in \eqref{eq:char1} -- \eqref{eq:char4}), where we have set 
$\X_k^t(x,v)=\X_k(t;0,x,v)$ and $\V_k^t(x,v)=\V_k(t;0,x,v)$ for simplicity. We therefore have the following explicit formula for $k=1,2$:
\begin{align}\label{eq:charfk}
f_k(t,\Z_k^t(x,v)) =e^{2t} f_0(x,v).
\end{align}
Following Loeper, one defines then
\begin{equation}
Q(t) := \int_{\R^2}\int_{\R^2} f_0(x,v) |\Z^t_1(x,v)-\Z^t_2(x,v)|^2 \dd v\,\dd x.
\label{eq:Loeper}
\end{equation}

In order not to pollute the estimates with terms that we cannot handle we should verify that the evolution of $Q$ is well-behaved. This will be true locally on an interval which is given by the following lemma (recall that $u_1,u_2\in\Ll^2(\R_+;\Ld^\infty(\R^2))$ by Lemma \ref{raduc}) 
\begin{lem}\label{lem:trajec-bis}
There exists $T_0\in [0,T]$, depending  only on $\|u_k\|_{\Ld^2(0,T;\Ld^\infty(\R^2))}$ for $k=1,2$,  such that for any $t\in [0,T_0]$ 
\begin{align}
\label{ineq:Z}
 \sup_{(t,x,v)\in [0,T_0]\times \R^2\times \mathbb{R}^2}|\Z_1^t(x,v)-\Z_2^t(x,v)|< \frac{e^{-1}}{2}\text{ and } \sup_{[0,T_0]} Q< \frac{e^{-1}}{2}.
\end{align}
 \end{lem}

\begin{rem}\label{rem:normalized}
Without the normalization assumption \eqref{ass:f0norm}, $T_0$ would also depend on $\|f_0\|_1$.
\end{rem}

\begin{proof}
From the definition of the characteristics, we first infer 
\begin{align*}
\V_k^t(x,v)  = v e^{-t}  + \int_0^t e^{-(t-\tau)} u_k(\tau,\X_k^\tau(x,v))\dd \tau,
\end{align*}
from which we have 
\begin{align*}
 |\V_1^t(x,v) - \V_2^t(x,v)|& \leq  \int_0^t( \|u_1(\tau)\|_{\infty}  +  \|u_2(\tau)\|_{\infty}) \,\dd \tau \\
&\leq t^{1/2}(\|u_1\|_{{\Ld^2(0,T;\Ld^\infty(\R^2))}}+\|u_2\|_{\Ld^2(0,T;\Ld^\infty(\R^2))})
\end{align*}
which leads after direct integration to 
\begin{align*}
 |\X_1^t(x,v) - \X_2^t(x,v)| \lesssim    t^{3/2}(\|u_1\|_{\Ld^2(0,T;\Ld^\infty(\R^2))}+\|u_2\|_{\Ld^2(0,T;\Ld^\infty(\R^2))}).
\end{align*}
Choosing $T_0$ sufficiently small, we get the first estimate and also the second one independently of $f_0$ (recall the normalized assumption \eqref{ass:f0norm} on $f_0$).  
\end{proof}

The dynamics of $Q$ is then controlled thanks to the following lemma 
\begin{lem}\label{lem:dyna-Q}
Let $T_0>0$ given by Lemma \ref{lem:trajec-bis}. There exists some $K>0$ and $\gamma\in\Ll^1(\R_+)$ such that for any $t\in[0,T_0]$ there holds
\begin{align*}
Q'(t) \leq 2Q(t)  + \gamma(t) \Psi(Q(t)) + K \|w(t)\|_2^2,
\end{align*}
where $\Psi$ is defined in \eqref{Psi}.
\end{lem}
\begin{proof}
For the sake of clarity, we omit the variables $(x,v)$ when there is no ambiguity in the subsequent integrands. Using the definition of the characteristics curves we infer 
\begin{align*}
\frac{1}{2}Q'(t) &= \int_{\R^2} \int_{\R^2} f_0 (\X^t_1-\X^t_2) \cdot (\V^t_1-\V^t_2) \dd v\,\dd x \\
&+ \int_{\R^2} \int_{\R^2} f_0 (\V^t_1-\V^t_2)\cdot(u_1(\X_1^t)-u_2(\X_2^t))  \dd v\,\dd x\\
&- \int_{\R^2} \int_{\R^2} f_0 |\V^t_1-\V^t_2|^2  \dd v\,\dd x,
\end{align*}
so that using Young's inequality and the non-negativity of $f_0$ we have 
\begin{align}
\nonumber Q'(t) \leq2 Q(t) &+2 \int_{\R^2}\int_{\R^2} f_0{|\V^t_1-\V^t_2| |u_1(\X_1^t)-u_1(\X_2^t)|}  \dd v\,\dd x\\
\label{ineq:Q'}&+ \int_{\R^2}\int_{\R^2} f_0 |u_1(\X_2^t)-u_2(\X_2^t)|^2  \dd v\,\dd x.
\end{align}
Consider $\gamma_1\in\Ll^1(\R_+)$ given by Lemma \ref{raducbis} for $u_1$. We infer from the definition of $T_0$ in Lemma \ref{lem:trajec-bis} the following estimate on $[0,T_0]$:
$$|u_1(\X_1^t)-u_1(\X_2^t)|\leq \gamma_1(t)\Psi(|\X_1^t-\X_2^t|),$$
so that we have for the first integral in the r.h.s. of \eqref{ineq:Q'} ($\Psi$ is non-decreasing):
\begin{align*}
\int_{\R^2}\int_{\R^2} f_0|\V^t_1-\V^t_2| |u_1(\X_1^t)-u_1(\X_2^t)| \, \dd v\,\dd x
\leq \gamma_1(t)\int_{\R^2}\int_{\R^2} f_0|\Z^t_1-\Z^t_2| \Psi(|\Z_1^t-\Z_2^t|) \,\dd v\,\dd x.
\end{align*} 
But one checks easily that $\tau \Psi(\tau)\leq \Psi(\tau^2)$ for $\tau\in[0,e^{-1}]$ and since $\Psi$ is concave, we infer from Jensen's inequality (recall the normalization \eqref{ass:f0norm} of $f_0$) 
\begin{align*}
\int_{\R^2}\int_{\R^2} f_0|\V^t_1-\V^t_2| |u_1(\X_1^t)-u_1(\X_2^t)| \, \dd v\,\dd x &\leq  \gamma_1(t)\int_{\R^2}\int_{\R^2} f_0(x,v)\Psi( |\Z_1^t-\Z_2^t|^2)   \dd v\,\dd x\\
&\leq \gamma_1(t)\Psi( Q(t)).
\end{align*}
For the second integral in the r.h.s. we write back the variables $(x,v)$ to perform the following change of variables $(x',v')=(\X_2^t(x,v),\V_2^t(x,v))$ (recall formula \eqref{eq:charfk} for $f_2$)
\begin{align*}
&  \int_{\R^2}\int_{\R^2} f_0(x,v) |u_1(t,\X_2^t(x,v))-u_2(t,\X_2^t(x,v))|^2  \dd v\,\dd x\\
&=\int_{\R^2}\int_{\R^2} f_2(t,x',v') |u_1(t,x')-u_2(t,x')|^2  \dd v'\,\dd x' \\
&= \int_{\R^2} \rho_2(t,x') |u_1(t,x')-u_2(t,x')|^2\dd x',
\end{align*}
and we use eventually the crucial estimate \eqref{ineq:crucial} to control $\rho_2$ uniformly : 
\begin{align*} 
\int_{\R^2}\int_{\R^2} f_0 |u_1(t,\X_2^t)-u_2(t,\X_2^t)|^2  \dd v\,\dd x
\leq \|\rho_2(t)\|_{\infty}\|w(t)\|_{2}^2
&\leq K\|w(t)\|_{2}^2,
\end{align*}
Gathering  in \eqref{ineq:Q'} the estimates that we obtained before, we get the result.
\end{proof}

Since $\Psi^{-1}$ is not integrable near $0$, the Osgood uniqueness criterion together with Lemma \ref{lem:dyna-Q} will ensure  that the addition of $Q$ in our energy functional does not deteriorate it. More precisely, recalling \eqref{ineq:w:gron} we have obtained for $t\in[0,T_0]$
\begin{multline*}
\|w(t)\|_{2}^2+Q(t) + \frac{3}{4}\int_0^t \|\nabla w(s)\|_{2}^2\,\dd s \\\leq K \int_0^t   \gamma(s)\|w(s)\|_2^2 \dd s + \int_0^t \gamma(s)(Q(s)+\Psi(Q(s))\,\dd s + \int_0^t A(s) \,\dd s ,
\end{multline*}
where $\gamma\in \Ll^1(\R_+)$. Since $\Psi$ is increasing, introducing $H(t):=\|w(t)\|_2^2 + Q(t)$ we can rewrite the previous estimate like this: 
\begin{multline}\label{ineq:gronwall-1}
H(t) + \frac{3}{4}\int_0^t \|\nabla w(s)\|_{2}^2\,\dd s \leq K\int_0^t   \gamma(s) \Big(H(s)+ \Psi(H(s))\Big)\,\dd s + \int_0^t A(s) \,\dd s.
\end{multline}

\medskip


\subsubsection{Step 3: Use of the Hardy's Maximal function}
In order to use a Gronwall estimate in \eqref{ineq:gronwall-1}, we have to show that $A(s)$ can be controlled in term of $Q(s)$, $\|w(s)\|_2^2$ and a 
small part of $\|\nabla w(s)\|_2^2$. In this paragraph all the estimates are established for $t\in[0,T_0]$, thus allowing the use of Lemma \ref{lem:trajec-bis} and Lemma \ref{lem:dyna-Q}.


We first rewrite $A(t)$ in terms of $f_1-f_2$ in the following way,
\begin{align*}
A(t)=\int_{\R^2} \int_{\R^2} (f_1-f_2)(t,x',v')  w(t,x') \cdot (v-u_2(t,x')) \, \dd v' \,\dd x'.
\end{align*}
We use once more the change of variables  $(x',v') =( \Z_k^t)^{-1}(x,v)$ ((recall formula \eqref{eq:charfk} for $f_1$ and $f_2$) and avoid specifying all the variables for legibility in the resulting integral:
\begin{align*}
A(t)&=\int_{\R^2} \int_{\R^2} f_0(x,v) \Big[w(\X_1^t) \cdot (\V_1^t-u_2(\X_1^t))-w(\X_2^t) \cdot (\V_2^t-u_2(\X_2^t))\Big] \, \dd v\, \dd x \\
&= \int_{\R^2} \int_{\R^2} f_0(x,v) w(\X^t_1)\cdot (\V^t_1-\V^t_2+u_2(\X_2^t)-u_2(\X^t_1))\,\dd v\,\dd x\\
&+\int_{\R^2} \int_{\R^2} f_0(x,v) \left[w(\X^t_1)-w(\X^t_2)\right]\cdot (\V^t_2-u_2(\X_2^t))\,\dd v\,\dd x\\
&=: A_1(t)+A_2(t).
\end{align*}
On the one hand, we have by applying  Lemma \ref{lem:trajec-bis} and Lemma~\ref{raducbis}  to $u_2$ (and noting $\gamma_2$ the corresponding $\Ll^1(\R_+)$ function) and the estimate \eqref{ineq:crucial} for the last line 
\begin{equation*}\begin{split}
|A_1(t)|& \leq \frac{1}{2} \int_{\R^2} \int_{\R^2} f_0(x,v) |w(\X_1^t)|^2 \,\dd v\,\dd x 
+  \frac{1}{2} \int_{\R^2} \int_{\R^2} f_0(x,v)|\V_1^t-\V_2^t|^2 \,\dd v\,\dd x \\ 
&+ \int_{\R^2} \int_{\R^2} f_0(x,v) |w(\X_1^t)||u_2(\X_2^t)-u_2(\X_1^t)| \,\dd v\,\dd x\\
&\leq \frac{1}{2}\|\rho_1(t)\|_\infty \|w(t)\|_2^2+ \frac{1}{2}Q(t)+
\gamma_2(t) \int_{\R^2} \int_{\R^2} f_0(x,v) |w(\X_1^t)|\Psi(|\Z_2^t-\Z_1^t|) \,\dd v\,\dd x,\\
&\leq K \Big(\|w(t)\|_2^2+Q(t)\Big)+
\gamma_2(t) \int_{\R^2} \int_{\R^2} f_0(x,v) |w(\X_1^t)|\Psi(|\Z_2^t-\Z_1^t|) \,\dd v\,\dd x.
\end{split}\end{equation*}
Since $\Psi$ is increasing, we have, setting $$\Y(t):=|w(\X_1^t)|+|\Z_2^t-\Z_1^t|,$$ 
noticing that 
$\tau \Psi(\tau)\leq \tau^2+\Psi(\tau^2)$ for all $\tau\geq 0$ and
using Jensen's inequality (recall that $\Psi$ is concave and $f_0$ normalized by \eqref{ass:f0norm}), that
\begin{equation*}\begin{split}
 \int_{\R^2} \int_{\R^2} f_0(x,v)& |w(\X_1^t)|\Psi(|\Z_2^t-\Z_1^t|) \,\dd v\,\dd x
\\
&\leq \int_{\R^2} \int_{\R^2} f_0(x,v)\Y(t)\Psi(\Y(t)) \,\dd v\,\dd x\\
&\leq  \int_{\R^2} \int_{\R^2} f_0(x,v)\Psi((\Y(t)^2) \,\dd v\,\dd x+\int_{\R^2} \int_{\R^2} f_0(x,v)\Y(t)^2 \,\dd v\,\dd x\\
&\leq \Psi\left(\int_{\R^2} \int_{\R^2} f_0(x,v)\Y(t)^2 \,\dd v\,\dd x\right)+\int_{\R^2} \int_{\R^2} f_0(x,v)\Y(t)^2 \,\dd v\,\dd x.
\end{split}\end{equation*}
On the other hand, we have, using \eqref{ineq:crucial} another time 
\begin{equation*}\begin{split}
\int_{\R^2} \int_{\R^2} f_0(x,v)\Y(t)^2 \,\dd v\,\dd x&\leq 2\int_{\R^2} \int_{\R^2} f_0(x,v) |w(\X_1^t)|^2\,\dd v\,\dd x+2Q(t)\\
&=2\int_{\R^2} \int_{\R^2} f_1(t,x',v') |w(t,x')|^2\,\dd v'\,\dd x'+2Q(t)\\
&\leq 2\|\rho_1(t)\|_{\infty}\|w(t)\|_2^2+2Q(t)\\
&\leq K(\|w(t)\|_2^2 + Q(t)).
\end{split}\end{equation*}
Finally we obtained the following estimate for $t\in [0,T_0]$:
\begin{align}
\label{ineq:A1}|A_1(t)| \leq K H(t) + \ffi_2(t) \left[K H(t) + \Psi(K H(t))\right],
\end{align}
where we recall that $H(t):=\|w(t)\|_2^2 + Q(t)$.

\vspace{2mm}

We can therefore focus on $A_2$, for which we are going to use properties of the maximal function of $w\in\H^1(\R^2)$, that are given in Proposition \ref{propo:max}. Recall the expression of $A_2$ : 
\begin{align*}
A_2(t)=\int_{\R^2} \int_{\R^2} f_0(x,v) \left[w(\X^t_1)-w(\X^t_2)\right]\cdot (\V^t_2-u_2(\X_2^t))\,\dd v\,\dd x.
\end{align*} 
We then invoke the a.e. pointwise estimate \eqref{ineq:max:2} to infer the following control for $A_2$:
\begin{align*}
|A_2(t)|\leq C \int_{\R^2}\int_{\R^2} f_0(x,v) |\X^t_1-\X^t_2| |\V_2^t -u_2(\X_2^t)| \Big[M |\nabla w|(\X^t_1) + M |\nabla w|(\X^t_2)\Big] \,\dd v\,\dd x.
\end{align*}
Therefore for all $\alpha>0$ and $t\in[0,T_0]$ (changing the value of $C$)
\begin{multline}
|A_2(t)|\leq  \frac{C}{\alpha}
\int_{\R^2} \int_{\R^2} f_0(x,v) |\X^t_1-\X^t_2|^2(1+|u_2(\X_2^t)|^2) \,\dd v\,\dd x \\
\label{ineq:A2}
 + \alpha \int_{\R^2}\int_{\R^2} f_0(x,v) \big(|\V^t_2|^2+1\big)\Big[(M |\nabla w|(\X^t_1))^2 + (M |\nabla w|(\X^t_2))^2\Big] \,\dd v\,\dd x.
\end{multline}
The first double integral in the right-hand side is bounded by 
$$ \frac{C}{\alpha} Q(t)(1+\|u_2(t)\|_\infty^2)\leq \frac{C}{\alpha} \gamma(t) Q(t),$$
where $\gamma \in \Ll^1(\R_+)$, recalling that $u_2\in\Ll^2(\R_+;\Ld^\infty(\R^2))$. For the second double integral, we notice first that (we use here \eqref{ineq:Z})
\begin{align*}
|\V^t_2|^2 \leq 2 (|\V^t_2-\V^t_1|^2 + |\V^t_1|^2 ) \leq 1 + 2|V_1^t|^2.
\end{align*}
In particular, after the usual change of variables $(x,v) = \Z_k^t(x',v')$ we get
\begin{multline*}
\int_{\R^2} \int_{\R^2} f_0(x,v)\big(|\V^t_2|^2+1\big)\Big[(M |\nabla w|(\X^t_1))^2 + (M |\nabla w|(\X^t_2))^2\Big] \,\dd v\,\dd x\\ \leq \int_{\R^2} \Big[2 m_2 f_1 + m_2 f_2 + 2 \rho_1 + \rho_2\Big](M|\nabla w|)^2 \,\dd x'.
\end{multline*}
Using estimate \eqref{ineq:crucial}, we obtain from \eqref{ineq:A2} the following inequality for $A_2$
\begin{align*}
|A_2(t)|\leq \frac{C}{\alpha} \gamma(t) Q(t) + K \alpha \int_{\R^2} (M|\nabla w)|)^2 \dd x',
\end{align*}
 and using  \eqref{ineq:max} (that is: the $\Ld^2(\R^2)$ continuity of the maximal function) we get eventually for $t\in[0,T_0]$
\begin{align*}
|A_2(t)| \leq \frac{K}{\alpha} \gamma(t) Q(t) + K \alpha \|\nabla w(t)\|_2^2.
\end{align*}
Summing the previous inequality to \eqref{ineq:A1} and changing the definition of $K>0$ and $\gamma\in\Ll^1(\R_+)$, we have for any $t\in[0,T_0]$ and any $0<\alpha<1$
\begin{align}\label{ineq:A}
|A(t)| \leq \frac{K}{\alpha} \gamma(t) H(t) + \gamma(t) \Psi( K H(t)) + K\alpha \|\nabla w(t)\|_2^2,
\end{align}
where we recall $H(t) = \|w(t)\|_2^2 + Q(t)$.
\subsubsection{Step 4: end of the proof}
Without loss of generality we can assume in the previous computations that the constant $K$ is greater that $1$ so that $\Psi(\cdot)\leq \Psi(K\cdot)$ (because $\Psi$ is non-decreasing). Using \eqref{ineq:A} in \eqref{ineq:gronwall-1}, we have thus obtained the existence of $K>1$ and $\gamma\in\Ll^1(\R_+)$ such that, for any $0<\alpha<1$ and for any $t\in [0,T_0]$
\begin{align*}
H(t) + \frac{3}{4}\int_0^t \|\nabla w(s)\|_{2}^2\,\dd s \leq \frac{K}{\alpha}\int_0^t   \gamma(s) \Big(H(s)+ \Psi( K H(s))\Big)\,\dd s + K\alpha \int_0^t \|\nabla w(s)\|_2^2\,\dd s.
\end{align*}
In particular, for $\alpha$ small enough so that $K \alpha \leq 3/K$, since 

$$
\forall a >0,\quad \int_0^a \frac{\dd s}{\Psi( K s)} = +\infty,
$$
we can conclude by Osgood's uniqueness criterion that 
$$\|w(t)\|_{2}^2+Q(t)\equiv 0\quad \text{on }[0,T_0),$$ that is: uniqueness holds on $[0,T_0]$. But $T_0$ depends only on $\|u_k\|_{\Ld^2(0,T;\Ld^\infty(\R^2))}$ for $k=1,2$, we can repeat repeating the argument a finite number of times, and recover the uniqueness on $[0,T]$.

\begin{rem}
Following Remark \ref{rem:normalized}, if \eqref{ass:f0norm} is not satisfied, $T_0$ depends also on $\|f_0\|_1$. The proof is however identical because the renormalized solutions that we consider conserves the integral value of the initial data at later times (this can be proven in a similar way to what we did for the propagation of second moment in paragraph \ref{subsubsec:moments}). Since here $f$ is non-negative we in particular have $\|f(t)\|_1 = \|f_0\|_1$ and the argument can thus be iterated up to $[0,T]$.
\label{rem:normalized2}
\end{rem}

\section{Comments and perspectives}
\label{sec: commentsperspectives}

In this work we have proved that any weak solution of the Vlasov-Navier-Stokes system satisfying the decay hypothesis (\ref{assu:lp}) must be unique. In doing so, the introduction of the Hardy's maximal function in Step 3 of Section \ref{sec: strategydetailed} is the key new tool allowing to close the proof, as a direct approach as in \cite{Loeper} seems impossible. Whether the mentioned assumption can be relaxed or not seems an open question. \par 

The methods developed in this paper might be adapted to derive some suitable stability estimates for the Vlasov-Navier-Stokes system, possibly choosing a slightly different functional $Q$ in the Step 2 of Section \ref{sec: strategydetailed}. \par  

A very interesting question arises in the 3D case, in which the uniqueness issue for Navier-Stokes alone proves already to be much subtler. We hope that under suitable stronger regularity assumptions, 
some uniqueness results 
might be achieved.

\begin{appendix}

\section{The Vlasov-Navier-Stokes system in $\R^2$}
\label{sec: appendixWeak}

Since the existence of global weak solutions for the Vlasov-Navier-Stokes system on $\R_+\times\R^2$ has not been written explicitly in the literature, we give here a (rather sketchy) overview of the proof following the one used in \cite{bougramou}. 

\vspace{2mm} 

First we replace the Vlasov-Navier-Stokes system by the following approximated one (with smooth, compactly supported initial data)  
\begin{align}
\label{app:u}\partial_t u + (u\star \ffi\cdot \nabla)u -\Delta u + \nabla p &=  \gamma\int_{\R^2}  \chi(v-u)f, \\
\label{app:divu}\div u & = 0, \\
\label{app:f}\partial_t f + v \cdot \nabla_x f + \nabla_v \cdot[f \chi(u-v)] &=0,
\end{align}
where $\ffi\in\mathscr{D}(\R^2)$ is a test function ($\star$ is the convolution in the space variable only), $\gamma\in\mathscr{D}(\R^2)$ is a nonnegative cut-off function (which depends only on $x$)  and $\chi\in\mathscr{D}(\R^2)^2$ is an odd vector-valued truncation function such that $0 \leq z\cdot \chi(z)$  for any $z\in\R^2$. Existence for this approximated system can be easily obtained thanks to the following fixed-point procedure. We consider the space $E:=\Ld^2(0,T;\H_{\textnormal{div}}^1(\R^2))$ and we use Schaefer's fixed-point Theorem: if a continuous map $\Theta:E\times[0,1]\rightarrow E$ sends bounded sets on compact sets, vanishes on $E\times\{0\}$ and is such that the set of all fixed points of the familly $(\Theta(\cdot,\sigma))_\sigma$ is bounded, then $\Theta(\cdot,1)$ has at least one fixed-point. 

\vspace{2mm}

Starting from $u\in E$, one can first consider the unique weak solution $f_u$ of \eqref{app:f}  (with initial condition $f^0)$ as obtained by the DiPerna-Lions theory. Then, for any $\sigma \in[0,1]$, one defines $\Theta(\sigma,u) = \gamma \widetilde{u}$ where $\widetilde{u}$ is the solution of the following regularized Navier-Stokes equation (with initial condition $\sigma u^0$)
\begin{align}
\label{app:utilde}\partial_t \widetilde{u} + (\widetilde{u} \star \ffi) \cdot \nabla \widetilde{u} - \Delta \widetilde{u} +\nabla p &= \sigma \gamma \int_{\R^2} \chi(v-u) f_u,\\
\div \widetilde{u} &= 0.
\end{align}
Of course when $\sigma=0$, one has $\Theta(\sigma,u)=0$, by uniqueness of the solution to the regularized Navier-Stokes equation. 

\vspace{2mm}

\textbf{Compactness.} Let us check that $\Theta$ sends bounded sets on compact ones. Since $\chi\in\Ld^\infty(\R^2)^2$ and $f^0$ is compactly supported, $f_u$ is also compactly supported (independently of $u$). In particular the $\Ld^2(0,T;\Ld^2(\R^2))$ norm  of the drag force  
\begin{align*}
\sigma \gamma \int_{\R^2} \chi(v-u) f_u
\end{align*}
is bounded for all $u$ by a constant depending only on $f^0$ (because $\sigma \leq 1$). Standard results fot the regularized Navier-Stokes equation thus lead to the desired compactness, thanks to the cut-off function $\gamma$ at the last step to allow the use of compact Sobolev embeddings. 

\vspace{2mm} 

\textbf{Continuity.} This compactness property of $\Theta$ and the uniqueness of the solution of the regularized Navier-Stokes equation, give the continuity of $\Theta$. Indeed, if $(\sigma_n)_n \rightarrow \sigma$ and $(u_n)_n\rightarrow u$, DiPerna-Lions Theory ensures that $(f_{u_n})_n$ converges to $f_u$ in a sufficiently strong sense to pass to the limit in the drag force. By the aforementionned compactness property, $(\widetilde{u}_n)_n$ has a cluster point in $\Ld^2(0,T;\H_{\textnormal{loc}}^1(\R^2))$ which can only be the unique solution of the regularized Navier-Stokes equation with the following right hand side  
\begin{align*}
\sigma \gamma \int_{\R^2} \chi(v-u) f_u,
\end{align*}
that is $\widetilde{u}$. 

\vspace{2mm} 

\textbf{\emph{A priori} estimate.} If $u=\Theta(\sigma,u)$, that is $u=\gamma \widetilde{u}$, multiplying \eqref{app:utilde} by $\widetilde{u}$ and \eqref{app:f} by $\sigma |v|^2/2$ and integrating by parts leads to the energy estimate  \begin{align*}
\frac{\sigma}{2} M_2 f_u(t)  + \frac{1}{2}\|\widetilde{u}(t)\|_2^2  &+ \int_0^t \|\nabla \widetilde{u}(s)\|_2^2 \dd s + \sigma \int_0^t \int \chi(v-u)\cdot (v-u) f_u \\
 &= \frac{\sigma}{2} M_2 f^0 + \frac{\sigma}{2}\|u^0\|_2^2.
\end{align*}
Since $z \cdot \chi(z) \geq 0$, the previous equality ensures that any fixed-point of $\Theta(\sigma,\cdot)$ remains in a ball of $E$, the radius of which depending only on $f^0$ and $u^0$ (because $\sigma \leq 1$).    

\vspace{2mm} 

\textbf{Back to Vlasov-Navier-Stokes.} For the last step, one only needs to perform the following limits : $(\ffi_n)_n\rightharpoonup \delta$, $(\gamma_n)_n \rightarrow 1$ and $(\chi_n)_n \rightarrow \textnormal{Id}$. The corresponding solution $(u_n)_n$ is weakly compact in $\Ld^2(0,T;\H^1(\R^2))\cap \Ld^\infty(0,T;\Ld^2(\R^2))$ and strongly compact in $\Ld^2(0,T;\Ld_{\textnormal{loc}}^2(\R^2))$ thanks to Aubin-Lions Lemma. This strong convergence is transfered to $(f_{u_n})_n$ thanks to the DiPerna-Lions theory,  allowing to pass to the limit each nonlinearity when integrated against a compactly supported test function.

\section{Estimates for the maximal function on the torus}
\label{sec: AppendixTorus}

We explain here how to obtain an analogue of Proposition~\ref{propo:max} for functions defined on the torus $\T^2$ (the general case  of $\T^d$ for $d \in \N^*$ is actually identical).
Although this must certainly be very classical, we gather these remarks in this appendix for the record, since we have not been able to find them in the literature.

For $g\in \Ld^1(\T^2)$ the corresponding maximal function is defined as
\begin{align*}
M g(x):=\sup_{0<r\leq \sqrt{2}/2}\frac{1}{|\B_r(x)|}\int_{\B_r(x)}|g|(y)\,\dd y,
\end{align*}
where $\B_r$ stands for balls for the canonical geodesic distance on the torus.

\begin{propo}\label{propo:maxT}
 If $g\in\Ld^2(\T^2)$ then so is $M g$ and furthermore
\begin{align}\label{ineq:maxT}
\|M g\|_2 \lesssim \|g\|_2.
\end{align}
If furthermore $g\in\H^1(\T^2)$ then, for a.e. $x,y\in\T^2$  there holds 
\begin{align}\label{ineq:max:2T}
|g(x)-g(y)|\lesssim |x-y| ( M |\nabla g|(x)+M |\nabla g|(y)),
\end{align}
where $|x-y|$ has to be understood as the geodesic distance between $x$ and $y$.
\end{propo}

\begin{proof} 
The proof of \eqref{ineq:max:2T} is the same as in the whole space case, and  follows directly, \emph{mutatis mutandis}, from \cite[Lemmas 1--3]{AcerbiFusco}.

Let's prove \eqref{ineq:maxT}. Fix $g\in\Ld^2(\T^2)$. We identify $\T^2$ with $[0,1)^2$ with periodic boundary conditions, and endowed with the normalized Lebesgue measure. Consider $\tilde{g} \in \Ld^2(\R^2)$ constructed as $1+8+16$ copies of $g$: one in the canonical cell $[0,1)^2$, $8+16$ others in the neighbouring surrounding cells, and equal to 
$0$ elsewhere.
By~\eqref{ineq:max},
$$
\|M \tilde{g}\|_2 \lesssim \|\tilde{g}\|_2 \lesssim   \|g\|_2.
$$
We remark that 
$$
\|M \tilde{g}\|_{L^2([0,1)^2)} \leq \|M \tilde{g}\|_{2} 
$$
and by definition of the maximal function on the torus, we have
$$
\| M g\|_{2} \lesssim \|M \tilde{g}\|_{L^2([0,1)^2)},
$$
so that 
$$
\|M g\|_2 \lesssim \|g\|_2.
$$

\vspace{2mm}

%
\end{proof}


\end{appendix}

\bibliographystyle{abbrv}
\bibliography{vns_unique}

\end{document}